\documentclass{article}

\usepackage{amsmath,amssymb,amsthm}

\usepackage{natbib}

\newtheorem{theorem}{Theorem}
\newtheorem{proposition}[theorem]{Proposition}
\newtheorem{lemma}[theorem]{Lemma}
\newtheorem{corollary}[theorem]{Corollary}
\newtheorem{example}[theorem]{Example}
\newtheorem{definition}[theorem]{Definition}
\newtheorem{remark}[theorem]{Remark}

\newcommand{\R}{\mathbb{R}}

\newcommand{\T}{\mathbb{T}}
\newcommand{\g}{\mathcal{L}}

\newcommand{\Ga}{\gamma_1}
\newcommand{\Gb}{\gamma_2}

\begin{document}

\title{Delta-Nabla Optimal Control Problems\footnote{Preprint version 
of an article submitted 28-Nov-2009; revised 02-Jul-2010;
accepted 20-Jul-2010; for publication in \emph{Journal of Vibration and Control}.
This work was carried out at the University of Aveiro via
the FCT post-doc fellowship SFRH/BPD/48439/2008 (Girejko);
a project of the Polish Ministry of Science and Higher Education
``Wsparcie miedzynarodowej mobilnosci naukowcow'' (Malinowska);
and the project Portugal--Austin UTAustin/MAT/0057/2008 (Torres).
The good working conditions at the University of Aveiro
and the partial support of CIDMA are here gratefully acknowledged.}}

\author{EWA GIREJKO$^{1}$\\ \texttt{egirejko@ua.pt}
\and AGNIESZKA B. MALINOWSKA$^{1}$\\ \texttt{abmalinowska@ua.pt}
\and DELFIM F. M. TORRES$^{2}$\footnote{Corresponding author.}\\ \texttt{delfim@ua.pt}}

\date{$^1$Faculty of Computer Science\\
Bia{\l}ystok University of Technology\\
15-351 Bia\l ystok, Poland\\[0.3cm]
$^2$Department of Mathematics\\
University of Aveiro\\
3810-193 Aveiro, Portugal}

\maketitle


\begin{abstract}
We present a unified treatment to control problems
on an arbitrary time scale by introducing the study
of forward-backward optimal control problems.
Necessary optimality conditions for delta-nabla isoperimetric problems
are proved, and previous results in the literature
obtained as particular cases. As an application of the results
of the paper we give necessary and sufficient Pareto optimality
conditions for delta-nabla bi-objective optimal control problems.
\end{abstract}





\section{INTRODUCTION}

In order to deal with non-traditional applications in areas such as medicine,
economics, or engineering, where the system dynamics are described
on a time scale partly continuous and partly discrete,
or to accommodate non-uniform sampled systems, one needs to work with systems
defined on a so called time scale -- see, \textrm{e.g.}, [\citet{A:B:L:06}],
[\citet{A:U:08}], [\citet{Basia:post_doc_Aveiro:2}].
The optimal control theory on time scales was introduced in the beginning
of the XXI century in the simpler framework of the calculus of variations,
and is now a fertile area of research
in control and engineering [\citet{SSW}], [\citet{MyID:183}].
In the literature there are two different approaches
to the problems of optimal control on time scales: some authors use the delta calculus
[\citet{B:04}], [\citet{des:ts:cv}], [\citet{B:T:08}], [\citet{F:T:08}],
[\citet{comNatyBasia:infHorizon}], [\citet{AM:T:W}],
while others prefer the nabla methodology [\citet{A:T}], [\citet{A:B:L:06}],
[\citet{A:U:08}], [\citet{NM:T}].
In this paper we propose a simple and effective unification
of the delta and nabla approaches of optimal control on time scales.
More precisely, we consider the problem of minimizing or maximizing
a delta-nabla cost integral functional
\begin{equation}
\label{eq:0}
\mathcal{L}(y) = \gamma_1\int_a^b L_{\Delta}\left(t,y^\sigma(t),y^\Delta(t)\right) \Delta t
+ \gamma_2\int_a^b L_{\nabla}\left(t,y^\rho(t),y^\nabla(t)\right) \nabla t
\end{equation}
subject to given boundary conditions and an isoperimetric constraint
of the form
\begin{equation}
\label{eq:1}
\mathcal{K}(y) = k_1\int_a^b K_{\Delta}\left(t,y^\sigma(t),y^\Delta(t)\right) \Delta t
+k_2\int_a^b K_{\nabla}\left(t,y^\rho(t),y^\nabla(t)\right) \nabla t=k \, .
\end{equation}
Main results include Euler-Lagrange necessary optimality type conditions
for delta-nabla isoperimetric problems \eqref{eq:0}--\eqref{eq:1}
(see Section~\ref{isoperimetric}).
Isoperimetric problems have found a broad class of
important applications throughout the centuries.
Concrete isoperimetric problems in engineering have been
investigated by a number of authors -- \textrm{cf.}
[\citet{A:T:2}], [\citet{Curtis}], and references therein.
Here, as an application of our results, we obtain the recent results of
[\citet{A:T}], [\citet{A:B:L:06}], [\citet{B:04}], and [\citet{F:T:10}]
as straightforward corollaries.
In Section~\ref{bi-objective} we consider delta-nabla bi-objective problems.
Our more general approach to optimal control in terms of the delta-nabla
problem \eqref{eq:0}--\eqref{eq:1}  allows to obtain necessary
and sufficient conditions for Pareto optimality.
The results of the paper are illustrated by several examples.


\section{PRELIMINARIES}
\label{sec:prelim}

We assume the reader to be familiar with the calculus on time scales.
For an introduction to the subject we refer to
the seminal papers [\citet{A:H}] and [\citet{SH}],
the nice survey [\citet{A:B:O:P:02}], and the books
[\citet{B:P:01}], [\citet{B:P:03}], and [\citet{Lak:book}].

Throughout the whole paper we assume $\T$ to be a given time scale with $a, b \in \T$,
$a < b$, and we set $I:=[a,b]\cap\T$ for $[a,b]\subset\R$. Moreover, we define
$I^{\kappa}_{\kappa}:=I^{\kappa}\cap I_{\kappa}$ with the standard
notations $I^{\kappa} = I \setminus (\rho(b),b]$
and $I_{\kappa} = I \setminus [a,\sigma(a))$.

We recall some necessary results.
If $y$ is delta differentiable at $t\in\T$, then $y^\sigma(t) = y(t) +
\mu(t) y^\Delta(t)$; if $y$ is nabla differentiable at $t$, then
$y^\rho(t) = y(t) - \nu(t) y^\nabla(t)$.
If the functions $f,g : \mathbb{T}\rightarrow\mathbb{R}$ are delta
and nabla differentiable with continuous derivatives, then the
following formulas of integration by parts hold:
\begin{equation}
\label{intBP}
\begin{split}
\int_{a}^{b}f^\sigma(t) g^{\Delta}(t)\Delta t
&=\left.(fg)(t)\right|_{t=a}^{t=b}
-\int_{a}^{b}f^{\Delta}(t)g(t)\Delta t \, , \\
\int_{a}^{b}f(t)g^{\Delta}(t)\Delta t
&=\left.(fg)(t)\right|_{t=a}^{t=b}
-\int_{a}^{b}f^{\Delta}(t)g^\sigma(t)\Delta t \, , \\
\int_{a}^{b}f^\rho(t)g^{\nabla}(t)\nabla t
&=\left.(fg)(t)\right|_{t=a}^{t=b}
-\int_{a}^{b}f^{\nabla}(t)g(t)\nabla t \, ,\\
\int_{a}^{b}f(t)g^{\nabla}(t)\nabla t
&=\left.(fg)(t)\right|_{t=a}^{t=b}
-\int_{a}^{b}f^{\nabla}(t)g^\rho(t)\nabla t \, .
\end{split}
\end{equation}

The following fundamental lemma of the calculus of variations on
time scales, involving a nabla derivative and a nabla integral,
was proved in [\citet{NM:T}].

\begin{lemma}{\rm (The nabla Dubois-Reymond lemma -- \textrm{cf.} Lemma~14 of [\citet{NM:T}])}
\label{DBRL:n} Let $f \in C_{\textrm{ld}}(I, \mathbb{R})$. If
$$
\int_{a}^{b} f(t)\eta^{\nabla}(t)\nabla t=0 \quad {\mbox for \ all}
\quad \eta \in C_{\textrm{ld}}^1(I, \mathbb{R}) \quad \mbox{ such \
that}\quad  \eta(a)=\eta(b)=0 \, ,
$$
then $f(t) \equiv c$ for all $t\in I_\kappa$, where $c$ is a
constant.
\end{lemma}

Lemma~\ref{DBRL:d} is the analogous delta version of
Lemma~\ref{DBRL:n}.

\begin{lemma}{\rm (The delta Dubois-Reymond lemma -- \textrm{cf.} Lemma~4.1 of [\citet{B:04}])}
\label{DBRL:d} Let $g\in C_{\textrm{rd}}(I, \mathbb{R})$. If
$$\int_{a}^{b}g(t) \eta^\Delta(t)\Delta t=0  \quad
\mbox{for all $\eta\in C_{\textrm{rd}}^1(I, \mathbb{R})$ such that
$\eta(a)=\eta(b)=0$,}$$ then $g(t)\equiv c \mbox{ on $I^\kappa$
for some $c\in\mathbb \R$}.$
\end{lemma}

Proposition~\ref{prop:rel:der} gives a relationship between delta
and nabla derivatives.

\begin{proposition}{\rm (\textrm{cf.} Theorems~2.5 and 2.6 of [\citet{A:G:02}])}
\label{prop:rel:der} (i) If $f : \mathbb{T} \rightarrow \mathbb{R}$
is delta differentiable on $\mathbb{T}^\kappa$ and $f^\Delta$ is
continuous on $\mathbb{T}^\kappa$, then $f$ is nabla differentiable
on $\mathbb{T}_\kappa$ and
\begin{equation}
\label{eq:chgN_to_D}
f^\nabla(t)=\left(f^\Delta\right)^\rho(t) \quad
\text{for all } t \in \mathbb{T}_\kappa \, .
\end{equation}
(ii) If $f : \mathbb{T} \rightarrow \mathbb{R}$ is nabla
differentiable on $\mathbb{T}_\kappa$ and $f^\nabla$ is continuous
on $\mathbb{T}_\kappa$, then $f$ is delta differentiable on
$\mathbb{T}^\kappa$ and
\begin{equation}
\label{eq:chgD_to_N}
f^\Delta(t)=\left(f^\nabla\right)^\sigma(t)
\quad \text{for all } t \in \mathbb{T}^\kappa \, .
\end{equation}
\end{proposition}

\begin{proposition}{\rm (\textrm{cf.} Theorem~2.8 of [\citet{A:G:02}])}
\label{eq:prop} Let $a, b \in\mathbb{T}$ with $a \le b$ and let $f$
be a continuous function on $[a, b]$. Then,
\begin{equation*}
\begin{split}
\int_a^b f(t)\Delta t &= \int_a^{\rho(b)} f(t)\Delta t
+ (b - \rho(b))f^\rho(b) \, , \\
\int_a^b f(t)\Delta t &= (\sigma(a) - a) f(a)
+ \int_{\sigma(a)}^b f(t)\Delta t \, , \\
\int_a^b f(t)\nabla t &= \int_a^{\rho(b)} f(t)\nabla t
+ (b - \rho(b)) f(b) \, , \\
\int_a^b f(t)\nabla t &= (\sigma(a) - a) f^\sigma(a) +
\int_{\sigma(a)}^b f(t)\nabla t \, .
\end{split}
\end{equation*}
\end{proposition}

We end our brief review of the calculus on time scales with a
relationship between the delta and nabla integrals.

\begin{proposition}{\rm (\textrm{cf.} Proposition~7 of [\citet{G:G:S:05}])}
If function $f : \mathbb{T} \rightarrow \mathbb{R}$ is continuous,
then for all $a, b \in \mathbb{T}$ with $a < b$ we have
\begin{gather}
\int_a^b f(t) \Delta t = \int_a^b f^\rho(t) \nabla t \, , \label{eq:DtoN}\\
\int_a^b f(t) \nabla t = \int_a^b f^\sigma(t) \Delta t \, . \label{eq:NtoD}
\end{gather}
\end{proposition}


\section{MAIN RESULTS}

Let $\mathbb{T}$ be a given time scale with $a, b \in \mathbb{T}$, $a < b$,
and $\mathbb{T} \cap (a,b) \ne \emptyset$; $L_{\Delta}(\cdot,\cdot,\cdot)$
and $L_{\nabla}(\cdot,\cdot,\cdot)$ be two given smooth
functions from $\T \times \mathbb{R}^2$ to $\mathbb{R}$ and $\Ga,\Gb\in\R$.
Our results are trivially generalized for
admissible functions $y : \T\rightarrow\mathbb{R}^n$
but for simplicity of presentation
we restrict ourselves to the scalar case $n=1$.

\subsection{DELTA-NABLA ISOPERIMETRIC PROBLEMS}
\label{isoperimetric}

We consider the delta-nabla integral functional
\begin{equation*}
\label{eq:P}
\begin{split}
\g(y)
&= \Ga\int_a^b L_{\Delta}\left(t,y^\sigma(t),y^\Delta(t)\right) \Delta t
+
\Gb\int_a^b L_{\nabla}\left(t,y^\rho(t),y^\nabla(t)\right) \nabla t \, .
\end{split}
\end{equation*}
For brevity we introduce the operators
$[y]$ and $\{y\}$ defined by
\begin{equation*}
[y](t) = \left(t,y^\sigma(t),y^\Delta(t)\right)
\ \  \text{ and } \ \
\{y\}(t) = \left(t,y^\rho(t),y^\nabla(t)\right) \, .
\end{equation*}
Then we can write:
\begin{equation*}
\begin{split}
\g_\Delta(y) &= \int_a^b L_\Delta[y](t) \Delta t \, , \\
\g_\nabla(y) &= \int_a^b L_\nabla\{y\}(t) \nabla t \, ,\\
\g(y) &= \Ga\g_\Delta(y)+\Gb \g_\nabla(y)
=\Ga\int_a^b L_{\Delta}[y](t) \Delta t
+ \Gb\int_a^b L_\nabla\{y\}(t)  \nabla t  \, .
\end{split}
\end{equation*}

Let $\alpha$, $\beta$, $\Ga$, $\Gb$, $k$, $k_1$, and $k_2$ be given real numbers.
Let us denote by $C_{\diamond}^{1}(I,\R)$ the class of functions
$y : I\rightarrow\mathbb{R}$  with
$(|\Ga| + |k_1|) y^\Delta$ continuous on $I^\kappa$
and $(|\Gb| + |k_2|) y^\nabla$ continuous on $I_\kappa$.
We consider the question of finding $y \in C_{\diamond}^{1}(I,\R)$
that is a solution to the problem
\begin{equation}\label{problem:P:iso}
\text{extremize}\ \ \mathcal{L}(y) = \gamma_1\int_a^b L_{\Delta}[y](t) \Delta t
+ \gamma_2\int_a^b L_{\nabla}\{y\}(t) \nabla t
\end{equation}
subject to the boundary conditions
\begin{equation}
\label{bou:con:iso}
y(a) = \alpha \, , \quad y(b) = \beta \, ,
\end{equation}
and the isoperimetric constraint
\begin{equation}
\label{const}
\mathcal{K}(y) = k_1\int_a^b K_{\Delta}[y](t) \Delta t
+k_2\int_a^b K_{\nabla}\{y\}(t) \nabla t=k \, ,
\end{equation}
where $K_{\Delta}(\cdot,\cdot,\cdot)$
and $K_{\nabla}(\cdot,\cdot,\cdot)$ are given smooth
functions from $\T \times \mathbb{R}^2$ to $\mathbb{R}$.

Function $y \in C_{\diamond}^{1}(I,\R)$ is said to be \emph{admissible}
provided it satisfies conditions \eqref{bou:con:iso} and \eqref{const}.
We are interested to obtain necessary conditions for an
admissible function to be a local minimizer (or a local
maximizer) to problem \eqref{problem:P:iso}--\eqref{const}.

\begin{definition}
\label{def:minimizer}
We say that $\hat{y}\in C_{\diamond}^{1}(I,\R)$
is a local minimizer (respectively local
maximizer) to problem \eqref{problem:P:iso}--\eqref{const} if there
exists $\delta >0$ such that
$\mathcal{L}(\hat{y})\leq \mathcal{L}(y)$
(respectively $\mathcal{L}(\hat{y}) \geq \mathcal{L}(y)$)
for all admissible functions $y \in C_{\diamond}^{1}(I,\R)$
satisfying the inequality
$\parallel y - \hat{y}\parallel_{1,\infty} < \delta$, where
$\parallel y\parallel_{1,\infty}:=
\parallel y^{\sigma}\parallel_{\infty}
+ \parallel y^{\rho}\parallel_{\infty} + \parallel
y^{\Delta}\parallel_{\infty} + \parallel
y^{\nabla}\parallel_{\infty}$
with $\parallel y\parallel_{\infty} :=\sup_{t \in
I_{\kappa}^{\kappa}}\mid y(t) \mid$.
\end{definition}

Let $\partial_{i}K$ denote the standard
partial derivative of a function $K(\cdot,\cdot,\cdot)$
with respect to its $i$th variable, $i = 1,2,3$.
The following definition is motivated by the time scale
Euler-Lagrange equations proved in [\citet{china-Xuzhou}]
and [\citet{Bedlewo:2009}].

\begin{definition}
\label{nor:abnor}
We say that $\hat{y} \in C_{\diamond}^{1}(I,\R)$ is an \emph{extremal} of
$$
\mathcal{K}(y) = k_1\int_a^b K_{\Delta}[y](t) \Delta t+k_2\int_a^b K_{\nabla}\{y\}(t) \nabla t
$$
if $\hat{y}$ satisfies the following Euler-Lagrange delta-nabla integral equations:
\begin{multline*}
k_1\left(\partial_3
K_\Delta[\hat{y}](\rho(t))
-\int_{a}^{\rho(t)} \partial_2 K_\Delta[\hat{y}](\tau) \Delta\tau\right)\\
+  k_2\left(\partial_3
K_\nabla\{\hat{y}\}(t) -\int_{a}^{t} \partial_2
K_\nabla\{\hat{y}\}(\tau) \nabla\tau\right) = \text{const} \quad
\forall t \in I_\kappa \, ;
\end{multline*}
\begin{multline*}
k_1\left(\partial_3
K_\Delta[\hat{y}](t)
-\int_{a}^{t} \partial_2 K_\Delta[\hat{y}](\tau) \Delta\tau\right)\\
+ k_2\left(\partial_3
K_\nabla\{\hat{y}\}(\sigma(t)) -\int_{a}^{\sigma(t)} \partial_2
K_\nabla\{\hat{y}\}(\tau) \nabla\tau\right) = \text{const} \quad
\forall t \in I^\kappa \, .
\end{multline*}
An extremizer (\textrm{i.e.}, a local minimizer or a local
maximizer) to problem \eqref{problem:P:iso}--\eqref{const} that is
not an extremal of $\mathcal{K}$ in \eqref{const}
is said to be a normal extremizer;
otherwise (\textrm{i.e.}, if it is an extremal of $\mathcal{K}$), the
extremizer is said to be abnormal.
\end{definition}

\begin{remark}
The word \emph{extremal} means ``solution of the Euler-Lagrange necessary optimality conditions''.
An extremizer is an extremal; but an extremal is not necessarily
an extremizer (it is just a candidate to extremizer given
by the first order necessary conditions).
\end{remark}

Associated to problem \eqref{problem:P:iso}--\eqref{const}
we introduce the following notations:
\begin{equation}
\label{eq:H1}
\begin{array}{l}
H_{\Delta}[\hat{y},\lambda](t):=H_{\Delta}(t,\hat{y}^{\sigma}(t),\hat{y}^{\Delta}(t),\lambda)
:=\Ga L_{\Delta}[\hat{y}](t)-k_1\lambda K_{\Delta}[\hat{y}](t)\\
H_{\nabla}\{\hat{y},\lambda\}(t):=H_{\nabla}(t,\hat{y}^{\rho}(t),\hat{y}^{\nabla}(t),\lambda)
:=\Gb L_{\nabla}\{\hat{y}\}(t)-k_2\lambda K_{\nabla}\{\hat{y}\}(t) \, .
\end{array}
\end{equation}
We look to $H_{\Delta}$ and $H_{\nabla}$ as functions of four
independent variables, and we denote the partial derivatives of
$H_{\Delta}(\cdot,\cdot,\cdot,\cdot)$
and
$H_{\nabla}(\cdot,\cdot,\cdot,\cdot)$
with respect to their $i$th argument, $i = 1,2,3,4$, by
$\partial_{i}H_{\Delta}$ and $\partial_{i}H_{\nabla}$ respectively.

\begin{theorem}[Necessary optimality conditions for normal extremizers
of a delta-nabla isoperimetric problem]
\label{thm:mr}
If $\hat{y} \in C_{\diamond}^{1}(I,\R)$ is a normal extremizer to the isoperimetric problem
\eqref{problem:P:iso}--\eqref{const}, then there exists $\lambda\in\R$ such that $\hat{y}$ satisfies the
following delta-nabla integral equations:
\begin{multline}
\label{iso:EL1}  \partial_3 H_\Delta[\hat{y},\lambda](\rho(t))+\partial_3H_\nabla\{\hat{y},\lambda\}(t)\\
-\left(\int_{a}^{\rho(t)} \partial_2 H_\Delta[\hat{y},\lambda](\tau) \Delta\tau
+\int_{a}^t \partial_2 H_\nabla\{\hat{y},\lambda\}(\tau) \nabla\tau\right)=\text{const}\ \ \forall t \in I_\kappa\,;
\end{multline}
\begin{multline}
\label{iso:EL2}
\partial_3 H_\Delta[\hat{y},\lambda](t)+\partial_3H_\nabla\{\hat{y},\lambda\}(\sigma(t))\\
-\left(\int_{a}^{t} \partial_2 H_\Delta[\hat{y},\lambda](\tau) \Delta\tau
+\int_{a}^{\sigma(t)} \partial_2 H_\nabla\{\hat{y},\lambda\}(\tau)\nabla\tau\right)
=\text{const}\ \ \forall t \in I^\kappa \, ,
\end{multline}
where $H_{\Delta}$ and $H_{\nabla}$ are defined by \eqref{eq:H1}.
\end{theorem}

\begin{proof}
Consider a variation of $\hat{y}$, say $\bar{y}=\hat{y} +
\varepsilon_{1} \eta_{1}+\varepsilon_{2} \eta_{2}$, where
$\eta_{i}\in C_{\diamond}^{1}(I,\R)$
and $\eta_{i}(a)=\eta_{i}(b)=0$, $i\in \{1,2\}$,
and $\varepsilon_{i}$ is a
sufficiently small parameter ($\varepsilon_{1}$ and
$\varepsilon_{2}$ must be such that $\parallel
\bar{y}-\hat{y}\parallel_{1,\infty}<\delta$ for some $\delta>0$).
Here $\eta_{1}$ is an arbitrary fixed function and $\eta_{2}$ is a
fixed function that will be chosen later. Define the real function
\begin{equation*}
\bar{K}(\varepsilon_{1},\varepsilon_{2})=\mathcal{K}(\bar{y})=k_1\int_a^b
K_{\Delta}[\bar{y}](t) \Delta t+ k_2\int_a^b
K_{\nabla}\{\bar{y}\}(t) \nabla t-k.
\end{equation*}
We have
\begin{multline*}
\left.\frac{\partial\bar{K}}{\partial
\varepsilon_{2}}\right|_{(0,0)} =
k_1\int_a^b \left(\partial_2 K_\Delta[\hat{y}](t) \eta_{2}^\sigma(t) +
\partial_3 K_\Delta[\hat{y}](t) \eta_{2}^\Delta(t)\right) \Delta t\\
+ k_2\int_a^b \left(\partial_2
K_\nabla\{\hat{y}\}(t) \eta_{2}^\rho(t) + \partial_3
K_\nabla\{\hat{y}\}(t) \eta_{2}^\nabla(t)\right) \nabla t\, .
\end{multline*}
The first and third integration by parts formula in \eqref{intBP} give
\begin{equation*}
\begin{split}
\int_a^b \partial_2 & K_\Delta[\hat{y}](t) \eta_{2}^\sigma(t)\Delta t\\
&= \int_a^t\partial_2 K_\Delta[\hat{y}](\tau)\Delta \tau \eta_{2}(t)|^{t=b}_{t=a}
-\int_a^b\left(\int_a^t
\partial_2 K_\Delta[\hat{y}](\tau)\Delta\tau\right)\eta_{2}^{\Delta}(t) \Delta t\\
&=-\int_a^b\left(\int_a^t \partial_2 K_\Delta[\hat{y}](\tau)\Delta\tau\right)\eta_{2}^{\Delta}(t) \Delta t
\end{split}
\end{equation*}
and
\begin{equation*}
\begin{split}
\int_a^b \partial_2 & K_\nabla\{\hat{y}\}(t) \eta_{2}^\rho(t)\nabla t\\
&=\int_a^t\partial_2 K_\nabla\{\hat{y}\}(\tau)\nabla\tau
\eta_{2}(t)|^{t=b}_{t=a}
-\int_a^b\left(\int_a^t\partial_2 K_\nabla\{\hat{y}\}(\tau)\nabla \tau \right) \eta_{2}^\nabla(t)\nabla t\\
&=-\int_a^b\left(\int_a^t\partial_2 K_\nabla\{\hat{y}\}(\tau)\nabla \tau \right) \eta_{2}^\nabla(t)\nabla t
\end{split}
\end{equation*}
since $\eta_{2}(a)=\eta_{2}(b)=0$. Therefore,
\begin{multline}
\label{after:parts}
\left.\frac{\partial\bar{K}}{\partial
\varepsilon_{2}}\right|_{(0,0)}= k_1\int_a^b
\left(\partial_3 K_\Delta[\hat{y}](t)-\int_a^t
\partial_2 K_\Delta[\hat{y}](\tau)\Delta\tau\right)\eta_{2}^{\Delta}(t) \Delta t\\
+k_2\int_a^b \left(\partial_3
K_\nabla\{\hat{y}\}(t)-\int_a^t\partial_2
K_\nabla\{\hat{y}\}(\tau)\nabla \tau \right)
\eta_{2}^\nabla(t)\nabla t.
\end{multline}
Let $$f(t)=\partial_3
K_\Delta[\hat{y}](t)-\int_a^t
\partial_2 K_\Delta[\hat{y}](\tau)\Delta\tau$$ and $$g(t)=\partial_3
K_\nabla\{\hat{y}\}(t)-\int_a^t\partial_2
K_\nabla\{\hat{y}\}(\tau)\nabla \tau .$$ We can then write
equation \eqref{after:parts} in the form
\begin{equation}\label{after:sub}
\left.\frac{\partial\bar{K}}{\partial
\varepsilon_{2}}\right|_{(0,0)}=k_1\int_a^bf(t)\eta_{2}^{\Delta}(t)
\Delta t+k_2\int_a^bg(t)\eta_{2}^\nabla(t)\nabla t.
\end{equation}
Transforming the delta integral in \eqref{after:sub} to a nabla
integral by means of \eqref{eq:DtoN}, we obtain that
\begin{equation*}
\left.\frac{\partial\bar{K}}{\partial
\varepsilon_{2}}\right|_{(0,0)}=k_1\int_a^b f^{\rho}(t)(\eta_{2}^{\Delta})^{\rho}(t)
\nabla t+k_2\int_a^bg(t)\eta_{2}^\nabla(t)\nabla t
\end{equation*}
and by \eqref{eq:chgN_to_D}
\begin{equation*}
\left.\frac{\partial\bar{K}}{\partial
\varepsilon_{2}}\right|_{(0,0)}=\int_a^b\left(k_1 f^{\rho}(t)
+k_2g(t)\right)\eta_{2}^\nabla(t)\nabla t.
\end{equation*}
As $\hat{y}$ is a normal extremizer, we conclude by
Lemma~\ref{DBRL:n} that there exists
$\eta_2$ such that $\left.\frac{\partial\bar{K}}{\partial
\varepsilon_{2}}\right|_{(0,0)}\neq 0$. Note that the same result
can be obtained by transforming the nabla integral in
\eqref{after:sub} to a delta integral by means of \eqref{eq:NtoD},
and then using Lemma~\ref{DBRL:d}. Since
$\bar{K}(0,0)=0$, by the implicit function theorem we conclude that
there exists a function $\varepsilon_{2}$ defined in the
neighborhood of zero such that
$\bar{K}(\varepsilon_{1},\varepsilon_{2}(\varepsilon_{1}))=0$, \textrm{i.e.},
we may choose a subset of variations $\bar{y}$ satisfying the
isoperimetric constraint. Let us now consider the real function
\begin{equation*}
\bar{L}(\varepsilon_{1},\varepsilon_{2})=\mathcal{L}(\bar{y})=\Ga\int_a^b
L_{\Delta}[\bar{y}](t) \Delta t+ \Gb\int_a^b
L_{\nabla}\{\bar{y}\}(t) \nabla t.
\end{equation*}
By hypothesis, $(0,0)$ is an extremal of $\bar{L}$ subject to the
constraint $\bar{K}=0$ and $\nabla \bar{K}(0,0)\neq \textbf{0}$. By
the Lagrange multiplier rule, there exists some real $\lambda$ such
that $\nabla(\bar{L}(0,0)-\lambda\bar{K}(0,0))=\textbf{0}$. Having
in mind that $\eta_{1}(a)=\eta_{1}(b)=0$, we can write
\begin{multline}
\label{function:L}
\left.\frac{\partial\bar{L}}{\partial
\varepsilon_{1}}\right|_{(0,0)} =
\Ga\int_a^b \left(\partial_3 L_\Delta[\hat{y}](t)-\int_a^t
\partial_2 L_\Delta[\hat{y}](\tau)\Delta\tau\right)\eta_{1}^{\Delta}(t) \Delta
t\\
+\Gb\int_a^b \left(\partial_3
L_\nabla\{\hat{y}\}(t)-\int_a^t\partial_2
L_\nabla\{\hat{y}\}(\tau)\nabla \tau \right)
\eta_{1}^\nabla(t)\nabla t
\end{multline}
and
\begin{multline}\label{function:K}
\left.\frac{\partial\bar{K}}{\partial
\varepsilon_{1}}\right|_{(0,0)}=k_1\int_a^b
\left(\partial_3 K_\Delta[\hat{y}](t)-\int_a^t
\partial_2 K_\Delta[\hat{y}](\tau)\Delta\tau\right)\eta_{1}^{\Delta}(t) \Delta
t\\
+k_2\int_a^b \left(\partial_3
K_\nabla\{\hat{y}\}(t)-\int_a^t\partial_2
K_\nabla\{\hat{y}\}(\tau)\nabla \tau \right)
\eta_{1}^\nabla(t)\nabla t.
\end{multline}
Let $$m(t)=\partial_3
L_\Delta[\hat{y}](t)-\int_a^t
\partial_2 L_\Delta[\hat{y}](\tau)\Delta\tau$$ and $$n(t)=\partial_3
L_\nabla\{\hat{y}\}(t)-\int_a^t\partial_2
L_\nabla\{\hat{y}\}(\tau)\nabla \tau .$$ Then equations
\eqref{function:L} and \eqref{function:K} can be written in the form
\begin{equation*}
\left.\frac{\partial\bar{L}}{\partial
\varepsilon_{1}}\right|_{(0,0)}=\Ga\int_a^bm(t)\eta_{1}^{\Delta}(t)
\Delta t+\Gb\int_a^bn(t)\eta_{1}^\nabla(t)\nabla t
\end{equation*}
and
\begin{equation*}
\left.\frac{\partial\bar{K}}{\partial
\varepsilon_{1}}\right|_{(0,0)}=k_1\int_a^bf(t)\eta_{1}^{\Delta}(t)
\Delta t+k_2\int_a^bg(t)\eta_{1}^\nabla(t)\nabla t.
\end{equation*}
Transforming the delta integrals in the above equalities to nabla
integrals by means of \eqref{eq:DtoN} and using \eqref{eq:chgN_to_D},
we obtain
\begin{equation*}
\left.\frac{\partial\bar{L}}{\partial
\varepsilon_{1}}\right|_{(0,0)}
=\int_a^b\left(\Ga m^{\rho}(t)+\Gb n(t)\right)\eta_{1}^\nabla(t)\nabla t
\end{equation*}
and
\begin{equation*}
\left.\frac{\partial\bar{K}}{\partial
\varepsilon_{1}}\right|_{(0,0)}=\int_a^b\left(k_1 f^{\rho}(t)
+ k_2 g(t)\right)\eta_{1}^\nabla(t)\nabla t.
\end{equation*}
Therefore,
\begin{equation}
\label{iso}
\int_{a}^{b}\eta_{1}^{\nabla}(t)\left\{\Ga m^{\rho}(t)+\Gb n(t)
-\lambda\left(k_1 f^{\rho}(t) + k_2 g(t)\right)\right\}\nabla t = 0.
\end{equation}
Since \eqref{iso} holds for any $\eta_{1}$,
by Lemma~\ref{DBRL:n} we have
\begin{equation*}
\Ga m^{\rho}(t)+\Gb n(t)-\lambda\left(k_1 f^{\rho}(t) + k_2 g(t)\right)=c
\end{equation*}
for some $c\in \mathbb{R}$ and all $t \in I_\kappa$. Hence,
condition \eqref{iso:EL1} holds. Equation \eqref{iso:EL1} can also
be obtained by transforming nabla integrals to delta integrals by
means of \eqref{eq:NtoD} and then using Lemma~\ref{DBRL:d}.
Equation \eqref{iso:EL2} can be shown in a totally analogous way.
\end{proof}

\begin{example}(normal extremals)
\label{example:iso}
(a) Let $\T=\{1,3,4\}$ and consider the problem
\begin{gather}
\label{ex:5}
\text{minimize }\ \g (y)=\int^4_1t\left(y^{\nabla}(t)\right)^2\nabla t \\
y(1)=0,\ \ y(4)=1
\end{gather}
subject to the constraint
\begin{equation}\label{ex:4}
\mathcal{K}(y)=\int^4_1t\left(y^{\Delta}(t)\right)^2\Delta t=\frac{105}{242}.
\end{equation}
Since $L_{\nabla}=t\left(y^{\nabla}\right)^2$ and
$K_{\Delta}=t\left(y^{\Delta}\right)^2$, we have
\[
\partial_2L_{\nabla}=0,\ \ \partial_3L_{\nabla}=2ty^{\nabla},\ \ \partial_2K_{\Delta}=0,\ \ \partial_3K_{\Delta}=2ty^{\Delta}.
\]
Let us assume for the moment that we are in conditions to apply Theorem~\ref{thm:mr}.
Applying equation~\eqref{iso:EL2} of Theorem~\ref{thm:mr}
we get the following delta-nabla differential equation:
\[
2\sigma(t)y^{\nabla}(\sigma(t))-\lambda2ty^{\Delta}(t)=C,\ \ t\in\{1,3\} \, ,
\]
where $C\in\R$. By \eqref{eq:chgD_to_N} we can write the above equation in the form
\begin{equation}\label{eq:iso1}
2\sigma(t)y^{\Delta}(t)-\lambda2ty^{\Delta}(t)=C,\ \ t\in\{1,3\}.
\end{equation}
Since $y^\Delta(1) = \left(y(3)-y(1)\right)/2 = y(3)/2$
and $y^\Delta(3) = y(4)-y(3) = 1-y(3)$, solving equation~\eqref{eq:iso1}
subject to the boundary conditions $y(1)=0$ and $y(4)=1$ we get
\begin{equation*}
\left\{
  \begin{array}{l}
    3y(3)-\lambda y(3)=C \\
    8(1-y(3))-6\lambda(1-y(3))=C,
  \end{array}
\right.
\end{equation*}
what implies
\begin{equation}\label{eq:iso2}
y(t)=\begin{cases}
        0 & \text{if $t=1$}\\
        \frac{8-6\lambda}{11-7\lambda} & \text{if $t=3$}\\
        1 & \text{if $t=4$.}
        \end{cases}
\end{equation}
Substituting \eqref{eq:iso2} into \eqref{ex:4} we obtain $\lambda_1=\frac{-11}{3}$,
$\lambda_2=\frac{143}{21}$. Hence, we get two extremals, $y_1$ and $y_2$,
corresponding to $\lambda_1$ and $\lambda_2$, respectively:
\[y_1(t)=\begin{cases}
        0 & \text{if $t=1$}\\
        \frac{9}{11} & \text{if $t=3$}\\
        1 & \text{if $t=4$}
        \end{cases},\,\ \
y_2(t)=\begin{cases}
        0 & \text{if $t=1$}\\
        \frac{69}{77} & \text{if $t=3$}.\\
        1 & \text{if $t=4$}
        \end{cases}
\]
One can easily check that $\g(y_1)=\frac{25}{22}$ and
$\g(y_2)=\frac{1345}{1078}$. We now show that $y_1$ is not an
extremal for $\mathcal{K}$. Indeed,
\begin{multline*}
\partial_3 K_\Delta[y_1](t)
-\int_{a}^{t} \partial_2 K_\Delta[y_1](\tau) \Delta\tau +
\partial_3 K_\nabla\{y_1\}(\sigma(t)) -\int_{a}^{\sigma(t)}
\partial_2 K_\nabla\{y_1\}(\tau) \nabla\tau\\
 =\partial_3
K_\Delta[y_1](t)=2ty_1^\Delta(t)=
\begin{cases}
        \frac{9}{11} & \text{if $t=1$}\\
        \frac{12}{11} & \text{if $t=3$}.
        \end{cases}
\end{multline*}
Thus $y_1$ is a candidate local minimizer to problem~\eqref{ex:5}--\eqref{ex:4}.\\
(b) Let $\T=\{1,3,4\}$ and consider the problem
\begin{gather}
\text{minimize }\ \g (y)=\int^4_1t\left(y^{\Delta}(t)\right)^2\Delta t\label{ex:5:b}\\
y(1)=0,\ \ y(4)=1 \label{ex:5:c}
\end{gather}
subject to the constraint
\begin{equation}
\label{ex:4:b}
\mathcal{K}(y)=\int^4_1t\left(y^{\nabla}(t)\right)^2\nabla t=\frac{25}{22}.
\end{equation}
Proceeding analogously as before, we find
\[y_1(t)=\begin{cases}
        0 & \text{if $t=1$}\\
        \frac{9}{11} & \text{if $t=3$}\\
        1 & \text{if $t=4$}
        \end{cases}\ \
        \]
as a candidate local minimizer to
problem~\eqref{ex:5:b}--\eqref{ex:4:b}.
\end{example}

As a particular case of Theorem~\ref{thm:mr} we obtain the following result:

\begin{corollary}[Necessary optimality condition for normal
extremizers of a delta isoperimetric problem -- \textrm{cf.} Theorem~3.4 of \citet{F:T:10}]
\label{T1}
Suppose that the problem of minimizing
\begin{equation*}
J(y)=\int_a^b L(t,y^\sigma(t),y^\Delta(t))\Delta t
\end{equation*}
subject to the boundary conditions $y(a)=y_a$, $y(b)=y_b$,
and the isoperimetric constraint
\begin{equation*}
I(y)=\int_a^b g(t,y^\sigma(t),y^\Delta(t))\Delta t=l
\end{equation*}
has a local solution at $\hat{y}$
in the class of functions $y:[a,b] \to \mathbb{R}$ such that $y^{\Delta}$
exists and is continuous on $[a,b]^{\kappa}$,
and that $\hat{y}$ is not an extremal for the functional $I$.
Then, there exists a Lagrange multiplier constant
$\lambda$ such that $\hat{y}$ satisfies
\begin{equation*}
\frac{\Delta}{\Delta t}\left[\partial_3 F(t,\hat{y}^\sigma(t),\hat{y}^\Delta(t))\right]
-\partial_2 F(t,\hat{y}^\sigma(t),\hat{y}^\Delta(t))=0\
\mbox{for all} \ t\in[a,b]^{\kappa^2}
\end{equation*}
with $F(t,x,v)=L(t,x,v)-\lambda g(t,x,v)$.
\end{corollary}

\begin{proof}
The result follows from Theorem~\ref{thm:mr} by considering
the particular case $\gamma_1 = k_1 = 1$ and $\gamma_2 = k_2 = 0$.
\end{proof}

One can easily cover abnormal extremizers within our result by
introducing an extra multiplier $\lambda_{0}$. Let
\begin{equation}
\label{eq:H2}
\begin{array}{l}
H_{\Delta}[\hat{y},\lambda_0,\lambda](t):=H_{\Delta}(t,y^{\sigma}(t),y^{\Delta}(t),\lambda_0,\lambda)
:=\Ga\lambda_0L_{\Delta}[\hat{y}](t)-k_1\lambda K_{\Delta}[\hat{y}](t)\\
H_{\nabla}\{\hat{y},\lambda_0,\lambda\}(t):=H_{\nabla}(t,y^{\rho}(t),y^{\nabla}(t),\lambda_0,\lambda)
:=\Gb\lambda_0L_{\nabla}\{\hat{y}\}(t)-k_2\lambda K_{\nabla}\{\hat{y}\}(t).
\end{array}
\end{equation}

\begin{theorem}[Necessary optimality conditions for normal and abnormal
extremizers of a delta-nabla isoperimetric problem]
\label{th:iso:abn}
If $\hat{y} \in C_{\diamond}^{1}(I,\R)$ is an extremizer to the isoperimetric problem
\eqref{problem:P:iso}--\eqref{const}, then there exist two constants
$\lambda_{0}$ and $\lambda$, not both zero, such that $\hat{y}$
satisfies the following delta-nabla integral equations:
\begin{multline}
\label{iso:EL1:abn}
\partial_3 H_{\Delta}[\hat{y},\lambda_0,\lambda](\rho(t))
+\partial_3 H_{\nabla}\{\hat{y},\lambda_0,\lambda\}(t)\\
-\int_{a}^{\rho(t)} \partial_2 H_{\Delta}[\hat{y},\lambda_0,\lambda](\tau) \Delta\tau
-\int_{a}^{t} \partial_2 H_{\nabla}\{\hat{y},\lambda_0,\lambda\}(\tau) \nabla\tau
= \text{const} \quad \forall t \in I_\kappa \, ;
\end{multline}
\begin{multline}
\label{iso:EL2:iso}
\partial_3 H_{\Delta}[\hat{y},\lambda_0,\lambda](t)
+ \partial_3 H_{\nabla}\{\hat{y},\lambda_0,\lambda\}(\sigma(t))\\
-\int_{a}^{t} \partial_2 H_{\Delta}[\hat{y},\lambda_0,\lambda](\tau) \Delta\tau
-\int_{a}^{\sigma(t)} \partial_2 H_{\nabla}\{\hat{y},\lambda_0,\lambda\}(\tau) \nabla\tau
= \text{const}\quad \forall t \in I^\kappa \, ,
\end{multline}
where $H_{\Delta}$ and $H_{\nabla}$ are defined by \eqref{eq:H2}.
\end{theorem}
\begin{proof}
Following the proof of Theorem~\ref{thm:mr}, since $(0,0)$ is an
extremal of $\bar{L}$ subject to the constraint $\bar{K}=0$, the
extended Lagrange multiplier rule (see for instance
Theorem~4.1.3 of [\citet{Brunt}]) asserts the existence of reals
$\lambda_{0}$ and $\lambda$, not both zero, such that
$\nabla(\lambda_{0}\bar{L}(0,0)-\lambda\bar{K}(0,0))=\textbf{0}$.
Therefore,
\begin{equation}
\label{iso:abn}
\int_{a}^{b}\eta_{1}^{\nabla}(t)\left\{\lambda_{0}\left(\Ga m^{\rho}(t)+\Gb n(t)\right)
-\lambda\left(k_1 f^{\rho}(t) + k_2 g(t)\right)\right\}\nabla t = 0.
\end{equation}
Since \eqref{iso:abn} holds for any $\eta_{1}$, by
Lemma~\ref{DBRL:n} we have
\begin{equation*}
\lambda_{0}\left(\Ga m^{\rho}(t)+\Gb n(t)\right)
-\lambda\left(k_1 f^{\rho}(t) + k_2 g(t)\right)=c
\end{equation*}
for some $c\in \mathbb{R}$ and all $t \in [a,b]_\kappa$. This
establishes equation \eqref{iso:EL1:abn}. Equation
\eqref{iso:EL2:iso} can be shown using a similar technique.
\end{proof}

\begin{remark}
If $\hat{y} \in C_{\diamond}^{1}(I,\R)$
is a normal extremizer to the isoperimetric problem
\eqref{problem:P:iso}--\eqref{const}, then we can choose $\lambda_{0}=1$
in Theorem~\ref{th:iso:abn} and obtain Theorem~\ref{thm:mr}. For
abnormal extremizers, Theorem~\ref{th:iso:abn} holds with
$\lambda_{0}=0$. The condition $(\lambda_{0},\lambda)\neq\textbf{0}$
guarantees that Theorem~\ref{th:iso:abn} is a useful necessary
optimality condition.
\end{remark}

\begin{example}(abnormal extremal)
\label{ex:iso:abnor}
Let $\T=\{1,3,4\}$ and consider the problem
\begin{gather}
\text{minimize }\ \g (y)=\int^4_1t\left(y^{\Delta}(t)\right)^2\Delta t\label{ex:5:ab}\\
y(1)=0,\ \ y(4)=1\label{ex:ab}
\end{gather}
subject to the constraint
\begin{equation}
\label{ex:4:ab}
\mathcal{K}(y)=\int^4_1t\left(y^{\nabla}(t)\right)^2\nabla t=\frac{12}{11}.
\end{equation}
Applying equation~\eqref{iso:EL1:abn} of Theorem~\ref{th:iso:abn} we
get the following delta-nabla differential equation:
\[
\lambda_0 2\rho(t)y^{\Delta}(\rho(t))-\lambda2ty^{\nabla}(t)=C,\ \
t\in\{3,4\} \, ,
\]
where $C\in\R$. By \eqref{eq:chgN_to_D} we can write the
above equation in the form
\begin{equation}\label{eq:iso1:ab}
\lambda_0 2\rho(t)y^{\nabla}(t)-\lambda2ty^{\nabla}(t)=C,\ \
t\in\{3,4\}.
\end{equation}
Substituting $t=3$ and $t=4$ into \eqref{eq:iso1:ab} we obtain
\begin{equation*}
\left\{
  \begin{array}{l}
    \lambda_0y(3)-3\lambda y(3)=C \\
    6\lambda_0(1-y(3))-8\lambda(1-y(3))=C.
  \end{array}
\right.
\end{equation*}
If we put $\lambda_0=1$, then the above system of equations has no
solutions. Therefore, we fix $\lambda_0=0$. In this case we obtain
\[y_0(t)=\begin{cases}
        0 & \text{if $t=1$}\\
        \frac{8}{11} & \text{if $t=3$}\\
        1 & \text{if $t=4$}
        \end{cases}\ \
        \]
as a candidate local minimizer to
problem~\eqref{ex:5:ab}--\eqref{ex:4:ab}. Observe that $y_0$ is an
extremal for $\mathcal{K}$. Indeed,
\begin{multline*}
\partial_3 K_\Delta[y_0](\rho(t))
-\int_{a}^{\rho(t)} \partial_2 K_\Delta[y_0](\tau) \Delta\tau \\
+ \partial_3 K_\nabla\{y_0\}(t) -\int_{a}^{t}
\partial_2 K_\nabla\{y_0\}(\tau) \nabla\tau
=2ty_0^\nabla(t)= \frac{24}{11},\quad t\in\{3,4\}.
\end{multline*}
\end{example}

As a particular case of Theorem~\ref{th:iso:abn} we obtain the main result of
[\citet{A:T}]:

\begin{corollary}[Necessary optimality condition for normal and abnormal
extremizers of a nabla isoperimetric problem -- \textrm{cf.} Theorem~2 of \citet{A:T}]
\label{thm:abn}
Let $\mathbb{T}$ be a time scale, $a,b \in \mathbb{T}$ with
$a < b$. If $\hat{y}$
is a local minimizer or maximizer to problem
\begin{gather*}
\text{extremize}\ \ \int_{a}^{b}f(t,y^\rho(t),y^\nabla(t))\nabla t\\
\int_{a}^{b}g(t,y^\rho(t),y^\nabla(t)) \nabla t =\Lambda \\
y(a)=\alpha\, , \quad y(b)=\beta
\end{gather*}
in the class of functions $y:[a,b] \to \mathbb{R}$ such that $y^{\nabla}$
exists and is continuous on $[a,b]_{\kappa}$,
then there exist two constants $\lambda_0$ and $\lambda$,
not both zero, such that
$$\frac{\nabla}{\nabla t} \left[\partial_3 G\left(t,\hat{y}^\rho(t),\hat{y}^\nabla(t)\right)\right]
-\partial_2 G \left(t,\hat{y}^\rho(t),\hat{y}^\nabla(t)\right)=0$$
for all $t \in [a,b]_{\kappa}$,
where $G(t,x,v)=\lambda_0 f(t,x,v)-\lambda g(t,x,v)$.
\end{corollary}
\begin{proof}
The result follows from Theorem~\ref{th:iso:abn} by considering
the particular case $\gamma_1 = k_1 = 0$ and $\gamma_2 = k_2 = 1$.
\end{proof}

Other interesting corollaries are easily obtained from Theorem~\ref{th:iso:abn}:

\begin{corollary}{\rm (The delta-nabla Euler-Lagrange equations
on time scales [\citet{china-Xuzhou}]).}
\label{thm:mr1}
If $\hat{y} \in C_{\diamond}^{1}(I,\R)$
is a local extremizer to problem
\begin{gather*}
\text{extremize}\ \ \mathcal{L}(y) = \Ga\int_a^b L_{\Delta}[y](t) \Delta t +
\Gb\int_a^b L_{\nabla}\{y\}(t) \nabla t \\
y(a) = \alpha \, , \quad y(b) = \beta \,\\
y \in C_{\diamond}^{1}(I,\R)\, ,
\end{gather*}
then $\hat{y}$ satisfies
the following delta-nabla integral equations:
\begin{multline}
\label{eq:EL1}
\gamma_1
\left(\partial_3 L_\Delta[\hat{y}](\rho(t))
-\int_{a}^{\rho(t)} \partial_2 L_\Delta[\hat{y}](\tau) \Delta\tau\right)\\
+
\gamma_2\left(\partial_3 L_\nabla\{\hat{y}\}(t)
-\int_{a}^{t} \partial_2 L_\nabla\{\hat{y}\}(\tau) \nabla\tau\right)
= \text{const}
\end{multline}
for all $t \in I_\kappa$; and
\begin{multline*}
\gamma_1\left(\partial_3 L_\Delta[\hat{y}](t)
-\int_{a}^{t} \partial_2 L_\Delta[\hat{y}](\tau) \Delta\tau\right)\\
+
\gamma_2\left(\partial_3 L_\nabla\{\hat{y}\}(\sigma(t))
-\int_{a}^{\sigma(t)} \partial_2 L_\nabla\{\hat{y}\}(\tau) \nabla\tau\right) = \text{const}
\end{multline*}
for all $t \in I^\kappa$.
\end{corollary}
\begin{proof}
The result follows from Theorem~\ref{th:iso:abn} by considering
the particular case $k_1 = k_2 = k = 0$,
for which the isoperimetric constraint \eqref{const} is trivially satisfied.
\end{proof}


\subsection{DELTA-NABLA BI-OBJECTIVE PROBLEMS}
\label{bi-objective}

We are now interested in studying the following bi-objective problem:
\begin{equation}\label{eq:13}
\text{minimize }\ F(y)=\left[
                      \begin{array}{l}
                        \g_\Delta(y)\\
                        \g_\nabla(y)
                      \end{array}
                    \right]
\end{equation}
with
\begin{equation*}
\begin{split}
\g_{\Delta}(y) &=\int_{a}^{b} L_{\Delta}(t,y^\sigma(t), y^{\Delta}(t))\Delta t=\int_{a}^{b}L_{\Delta}[y](t) \Delta t\, ,\\
\g_{\nabla}(y) &=\int_{a}^{b} L_{\nabla}(t,y^\rho(t), y^{\nabla}(t))\nabla t=\int_a^b L_\nabla\{y\}(t) \nabla t \, ,
\end{split}
\end{equation*}
and $y\in C^1_{\diamond}(I,\R)$, $y(a)=\alpha$, $y(b)=\beta$, $t\in I$.
A solution to this vector optimization problem
is understood in the Pareto sense.

\begin{definition}[locally Pareto optimal solution]
A function $\hat{y}\in C_{\diamond}^{1}(I,\R)$
is called \emph{a local Pareto optimal solution} if there exists
$\delta>0$ for which does not exist $y\in C_{\diamond}^{1}(I,\R)$
with $||\hat{y}-y||_{1,\infty}<\delta$ and
\begin{equation*}
\g_\Delta(y)\leq \g_\Delta(\hat{y})\ \
\wedge\ \ \g_\nabla(y)\leq\g_\nabla(\hat{y}),
\end{equation*}
where at least one of the above inequalities is strict.
\end{definition}

\begin{theorem}[Necessity]\label{th:nec Pareto}
If $\hat{y}$ is a local Pareto optimal solution to the
bi-objective problem~\eqref{eq:13}, then $\hat{y}$
is a minimizer to the isoperimetric problems
\begin{equation*}
\text{minimize}\ \ \g_\Delta(y)\ \ \text{subject to}\ \
\g_\nabla(y)=\g_\nabla(\hat{y})
\end{equation*}
and
\begin{equation*}
\text{minimize}\ \ \g_\nabla(y)\ \ \text{subject to}\ \
\g_\Delta(y)=\g_\Delta(\hat{y})
\end{equation*}
simultaneously.
\end{theorem}

\begin{proof}
A proof can be done similarly to the proof of
Theorem~3.8 in [\citet{AM:T}].
\end{proof}

\begin{example}
\label{ex:mobj1}
Let us consider $\T=\{1,3,4\}$ and the bi-objective
optimization problem \eqref{eq:13} with
\begin{equation}
\label{ex:pareto}
\begin{split}
\g_\Delta(y)&=\int^4_1t\left(y^\Delta(t)\right)^2\Delta t, \\
\g_\nabla(y)&=\int^4_1t\left(y^\nabla(t)\right)^2\nabla t.
\end{split}
\end{equation}
We pose the question of finding local Pareto optimal solutions
to \eqref{ex:pareto} under the boundary conditions
\begin{equation}
\label{ex:bc:pareto}
y(1)=0,\quad  y(4)=1.
\end{equation}
Let us consider the following function
\[\hat{y}(t)=\begin{cases}
        0 & \text{if $t=1$}\\
        \frac{9}{11} & \text{if $t=3$}\\
        1 & \text{if $t=4$.}
        \end{cases}\]
As it is shown in Example~\ref{example:iso},
$\hat{y}$ is, simultaneously, a candidate minimizer to the problem
\begin{gather*}
\text{minimize }\ \g_\nabla (y)=\int^4_1t\left(y^{\nabla}(t)\right)^2\nabla t \\
y(1)=0,\ \ y(4)=1
\end{gather*}
subject to
\begin{equation*}
\mathcal{L}_\Delta(y)=\int^4_1t\left(y^{\Delta}(t)\right)^2\Delta t=\frac{105}{242}.
\end{equation*}
and
\begin{gather*}
\text{minimize }\ \g_\Delta (y)=\int^4_1t\left(y^{\Delta}(t)\right)^2\Delta t \\
y(1)=0,\ \ y(4)=1
\end{gather*}
subject to
\begin{equation*}
\mathcal{L}_\nabla(y)=\int^4_1t\left(y^{\nabla}(t)\right)^2\nabla t=\frac{25}{22}\, .
\end{equation*}
According to Theorem~\ref{th:nec Pareto}, the function $\hat{y}$
is a candidate Pareto optimal solution
to the bi-objective problem \eqref{ex:pareto}--\eqref{ex:bc:pareto}.
\end{example}

Theorem~\ref{th:nec Pareto}
shows that necessary optimality conditions to
isoperimetric problems (see Section~\ref{isoperimetric}) are also
necessary to local Pareto optimality of a bi-objective variational
problem on time scales. Indeed, functional
\eqref{problem:P:iso} in particular cases when $\Ga=1$ and $\Gb=0$
or $\Ga=0$ and $\Gb=1$ is reduced  either to $\mathcal{L}(y)
= \g_\Delta(y)$ or to $\mathcal{L}(y) =\g_\nabla(y)$.

The next theorem asserts that sufficient conditions of optimality for scalar
optimal control problems are also sufficient conditions for Pareto optimality.

\begin{theorem}[Sufficiency]
\label{th:1}
A local minimizer $\hat{y}$ to the functional
$\gamma\g_\Delta(y)+(1-\gamma)\g_\nabla(y)$ with $\gamma\in(0,1)$
is a local Pareto optimal solution to the bi-objective problem \eqref{eq:13}.
\end{theorem}

\begin{proof}
A proof can be done similarly to the proof of Theorem~3.7 in [\citet{AM:T}].
\end{proof}

\begin{example}
\begin{description}
\item[a)]
Let us consider $\T=\{1,3,4\}$ and the bi-objective
optimization problem \eqref{eq:13} defined by
\begin{equation}\label{ex:pareto nec}
\begin{array}{c}
                \g_\Delta(y)=\int^4_1 t\left(y^\Delta(t)\right)^2\Delta t, \\
                \g_\nabla(y)=\int^4_1 t\left(y^\nabla(t)\right)^2\nabla t
              \end{array}
\end{equation}
subject to
\begin{equation}\label{ex:pareto:bc}
y(1)=0,\quad  y(4)=1.
\end{equation}
By Theorem~\ref{th:1} we can find Pareto optimal solutions to this
problem by considering the family of problems
\begin{gather*}
\min\ \ \gamma\g_\Delta(y)+(1-\gamma)\g_\nabla(y)\\
y(1)=0,\quad  y(4)=1,
\end{gather*}
where $\gamma\in(0,1)$. Using condition \eqref{eq:EL1} of
Corollary~\ref{thm:mr1} we get the following equation:
\begin{equation}
\label{fam:11}
2\gamma\rho(t)y^{\Delta}(\rho(t)) + 2(1-\gamma)ty^{\nabla}(t)
= c \quad \forall t \in \{3,4\}
\end{equation}
for some $c\in \R$. Substituting $t=3$ and $t=4$ into \eqref{fam:11}, we obtain
\begin{gather*}
\gamma y(3)+3(1-\gamma)y(3) = c,\\
6\gamma(1-y(3))+8(1-\gamma)(1-y(3)) = c\, ,
\end{gather*}
and from this we have $y(3)=\frac{8-2\gamma}{11-4\gamma}$, $\gamma\in (0,1)$.
Since $L_\Delta(\cdot,\cdot,\cdot)$ and $L_\nabla(\cdot,\cdot,\cdot)$
are jointly convex with respect to the second and third
argument for any $t\in \T$, the local Pareto
optimal solutions to problem \eqref{ex:pareto nec}--\eqref{ex:pareto:bc} are
\begin{equation*}
\begin{array}{c}
  y(t)=\left\{
         \begin{array}{lll}
           0, & \hbox{if $t=1$} \\
           k, & \hbox{if $t=3$,} & k\in \left(\frac{8}{11}, \frac{6}{7}\right).\\
           1, & \hbox{if $t=4$,}
         \end{array}
       \right.
\end{array}
\end{equation*}
  \item[b)] Let us consider $\T=\{0,1,2\}$ and the bi-objective
problem \eqref{eq:13} defined by
\begin{equation}\label{ex:pr}
\begin{array}{c}
                \g_\Delta(y)=\int^2_0\left(y^\sigma(t)\right)^2\Delta t, \\
                \g_\nabla(y)=\int^2_0\left(y^\rho(t)-3\right)^2\nabla
                t
              \end{array}
\end{equation}
subject to
\begin{equation}\label{ex:bc}
y(0)=0,\quad  y(2)=0.
\end{equation}
By Theorem~\ref{th:1} we can find Pareto optimal solutions to this
problem by considering the family of problems
\begin{gather*}
min\ \ \gamma\g_\Delta(y)+(1-\gamma)\g_\nabla(y)\\
y(0)=0,\quad  y(2)=0,
\end{gather*}
where $\gamma\in(0,1)$. Using condition \eqref{eq:EL1} of
Corollary~\ref{thm:mr1} we get the following equation:
\begin{equation}
\label{fam:1}
\gamma \int_{0}^{\rho(t)} y^\sigma(\tau) \Delta\tau
+(1-\gamma) \int_{0}^{t} (y^\rho(\tau)-3) \nabla\tau
= c \quad \forall t \in \{1,2\}
\end{equation}
for some $c\in \R$. Substituting $t=1$ and $t=2$
into \eqref{fam:1} we obtain
\begin{equation*}
\begin{split}
\gamma\int_{0}^{0} y^\sigma(\tau) \Delta\tau
+(1-\gamma)\int_{0}^{1} (y^\rho(\tau)-3) \nabla\tau = c,\\
\gamma\int_{0}^{1} y^\sigma(\tau) \Delta\tau
+(1-\gamma)\int_{0}^{2} (y^\rho(\tau)-3) \nabla\tau = c\, ,
\end{split}
\end{equation*}
and from this we have $y(1)=3-3\gamma$, $\gamma \in (0,1)$. Since
$L_\Delta(\cdot,\cdot,\cdot)$ and $L_\nabla(\cdot,\cdot,\cdot)$
are jointly convex with respect to the second and third
argument for any $t\in \T$, the local Pareto
optimal solutions to problem \eqref{ex:pr}--\eqref{ex:bc} are
\begin{equation*}
\begin{array}{c}
  y(t)=\left\{
         \begin{array}{lll}
           0, & \hbox{if $t=0$} \\
           k, & \hbox{if $t=1$,} & k \in (0,3).\\
           0, & \hbox{if $t=2$,}
         \end{array}
       \right.
\end{array}
\end{equation*}
\end{description}
\end{example}


\bigskip

\emph{Acknowledgments.} {\it The authors would like to express
their gratitude to two anonymous referees, for several relevant
and stimulating remarks contributing to improve the quality of the paper.}



\end{document}